\documentclass[leqno,a4paper,12pt]{amsart}

\usepackage{amsmath,amsfonts,amsthm,amssymb,dsfont}
\usepackage[alphabetic]{amsrefs}

\usepackage{times}

\usepackage{enumerate}

\usepackage{mathrsfs}
\usepackage{enumerate, xspace}
\usepackage{graphicx}
\usepackage{color}

    \usepackage[all, knot]{xy}
    \xyoption{arc} 

\usepackage{amssymb,amsmath,amsthm}

\usepackage{verbatim}

\usepackage{caption}
\usepackage{subcaption}

\DeclareMathOperator\g{girth}
\DeclareMathOperator\diam{diam}
\DeclareMathOperator\fold{Fold}

\usepackage[pdfauthor={Goulnara Arzhantseva and Markus Steenbock},pdftitle={Rips construction without unique product},pdfkeywords={Graphical small cancellation theory, unique product property, Kaplansky's zero-divisor conjecture, van Kampen diagrams, free products of groups},pdftex]{hyperref}

\hypersetup{colorlinks,%
citecolor=black,%
filecolor=black,%
linkcolor=black,%
urlcolor=black,%
}

\theoremstyle{definition}

\theoremstyle{plain}		\newtheorem{prop}{Proposition}		
				\newtheorem{cor}{Corollary}
				
				\newtheorem{thm}{Theorem}

				\newtheorem{pbm}{Open problem}
								\newtheorem{them}{Theorem}[section]
								\newtheorem{pr}{Proposition}[section]
								\newtheorem{lemma}{Lemma}[section]
\theoremstyle{remark}

\begin{document}
\author{Goulnara Arzhantseva}
\address{Universit\"at Wien, Fakult\"at f\"ur Mathematik\\
Oskar-Morgenstern-Platz 1, 1090 Wien, Austria.
}
\address{Erwin Schr\"odinger International Institute for Mathematical Physics\\
Boltzmanngasse~9, 1090 Wien, Austria.
}
\email{goulnara.arzhantseva@univie.ac.at}

\author{Markus Steenbock}
\address{Universit\"at Wien, Fakult\"at f\"ur Mathematik\\
Oskar-Morgenstern-Platz 1, 1090 Wien, Austria.
}
\email{markus.steenbock@univie.ac.at}

\title{Rips construction without unique product}

\begin{abstract}
Given a finitely presented group $Q,$  we produce a short exact sequence $1\to N \hookrightarrow G \twoheadrightarrow Q \to 1$ such that $G$ is a torsion-free Gromov hyperbolic group without the unique product property and $N$ is without the unique product property and has Kazhdan's Property (T). 
Varying $Q,$ we show a wide diversity of concrete examples of  Gromov hyperbolic groups without the unique product property.
As an application, we obtain Tarski monster groups without the unique product property.
\end{abstract}

\maketitle

\section{Introduction}

A group $G$ has the \emph{unique product property} (or said to be a \emph{unique product group}) whenever for all pairs of non-empty finite subsets $A$ and $B$ of $G$ the set of  \mbox{products} $AB$ has an element $g\in G$ with a unique representation
of the form $g=ab$ with $a\in A$ and $b\in B$. Unique product groups are torsion-free. They satisfy the \mbox{outstanding} Kaplansky zero-divisor conjecture \cites{kaplansky_problems_1957,kaplansky_problems_1970}, which states that the group ring of a torsion-free group over an integral domain has no zero-divisors \cite{cohen_zero_1974}. Rips and Segev~\cite{rips_torsion-free_1987} gave the first examples of torsion-free groups without the unique product property. In~\cite{markus}, the second author has generalized their examples, 
proved that the (generalized) Rips-Segev groups are Gromov hyperbolic, and provided an uncountable family of non unique product groups. Other examples of torsion-free groups without  unique product can be found~in~\cites{passman_algebraic_1977,promislow_simple_1988,carter_new_2013}.\smallskip

Our goal is to construct new concrete examples of non unique product groups with diverse algebraic and geometric properties.
We realize this by extending further our construction of generalized Rips-Segev groups and by showing that every finitely presented group is a non-trivial quotient of a torsion-free Gromov hyperbolic non unique product group.

\begin{thm}\label{main}
 Let $Q$ be  a finitely presented group. Then there exists a short exact  sequence \[1\to N \hookrightarrow G \twoheadrightarrow Q \to 1\] such that 
 \begin{itemize}
  \item  $G$ is a torsion-free Gromov hyperbolic group without the unique product property, 
  \item $N$ is a $2$-generated subgroup of $G$. 
 \end{itemize}
\end{thm}

The assumption on finite presentation of $Q$ can be relaxed and our method still provides a non unique product group $G$. 
In such a general setting, $G$ is not any more Gromov hyperbolic, although  it is a direct limit of those  (in fact, of graphical small cancellation groups, see more details in Section~\ref{sec:graphical}).

\begin{thm}\label{main2}
 Let $Q$ be  a finitely generated  group. Then there exists a short exact  sequence \[1\to N \hookrightarrow G \twoheadrightarrow Q \to 1\] such that 
 \begin{itemize}
  \item $G$ is a torsion-free group without the unique product property which is a direct limit of Gromov hyperbolic groups, 
  \item $N$ is a $2$-generated subgroup of $G$. 
 \end{itemize}
\end{thm}

Varying $Q$ in these theorems, we obtain many new groups without the unique product property that have various algebraic and algorithmic properties, see 
Section~\ref{sec:algo}.
 
 We extend our construction further and produce strongly non-amenable examples.

\begin{thm}
 Let $Q$ be  a finitely generated  group. Then there exists a short exact  sequence \[1\to N \hookrightarrow G \twoheadrightarrow Q \to 1\] such that 
 \begin{itemize}
  \item $G$ is a torsion-free group without the unique product property which is a direct limit of Gromov hyperbolic groups, 
  \item $N$ is a subgroup of $G$ with Kazhdan's Property (T) and without the unique product property. 
 \end{itemize}
\end{thm}

We provide, in particular, first examples of Property (T) groups without the unique product property. 

\begin{cor}\label{C: T}
 There are torsion-free Gromov hyperbolic groups with Kazhdan's Property (T) and without the unique product property. 
\end{cor}

Our approach combines three constructions: the famous Rips construction \cite{rips_subgroups_1982}, the construction by Rips-Segev of torsion-free groups 
without the unique product property~\cite{rips_torsion-free_1987}, and Gromov's groundbreaking construction of graphical small cancellation groups 
with Property (T), cf.~\cite{gromov_random_2003}*{1.2.A, 4.8.(3)}, based on his spectral characterization of this property~\cites{silberman_random_2003,ollivier_kazhdan_2007}. 

An essential technical point in our proofs is that we explain all three constructions using the graphical small cancellation theory over the free product of groups. 
We show, in particular, that Gromov's probabilistic construction of graphs defining groups with Property (T) is flexible under taking edge subdivisions. 
 \begin{thm}\label{thm:both} For all $m> 64$, there exists a  finite  connected graph $\mathcal{T}$ labeled by $\{a_1,\ldots, a_m\}
 \
 $ such that the labeling satisfies the $Gr_*'(1/6)$--small cancellation condition over the free product $\langle a_1 \rangle * \ldots * \langle a_m \rangle$, the labeling satisfies the $Gr'(1/6)$--small cancellation condition with respect to the word length metric, and the group with $a_1,\ldots, a_m$ as generators and the labels of the cycles of $\mathcal{T}$ as relators has Property~(T).
 \end{thm}
 
The graph $\mathcal{T}$ is produced by assigning to every edge of an expander graph a letter and an orientation independently uniformly at random.  It is an interesting technical outcome that the small cancellation conditions over the free group and over the free product  can be combined in such graphs, see Section~\ref{sec:ap}.  
This flexibility in the small cancellation condition is useful for constructing new groups with exotic properties. \smallskip

Our examples are in a strong contrast with the previously known constructions of torsion-free groups without the unique product property, 
alternative to the Rips-Segev groups \cites{passman_algebraic_1977,promislow_simple_1988,carter_new_2013}. Indeed,  
all those constructions yield infinite groups with the Haagerup property\footnote{Groups in \cites{passman_algebraic_1977,promislow_simple_1988} are solvable, hence, a-T-menable; groups in~\cite{carter_new_2013}  are a-T-menable as they have $\mathbb{Z}^k\times\mathbb{F}_m$ as a finite index subgroup.} (= a-T-menable groups, in the terminology of Gromov, see~\cite{ChCJJV}), and, hence, groups which do not  have Property~(T). \smallskip

An alternative proof of our Corollary \ref{C: T}, although with no probabilistic and, hence, genericity aspects underlying Theorem~\ref{thm:both}, can be obtained using the small cancellation theory over hyperbolic groups.

\begin{thm}[Ol'shanskii, cf. \cite{olshanskii_residualing_1993}*{Th. 2}]\label{olsh}
Let $G=H_1*H_2$ be the free product of two non-elementary torsion-free Gromov hyperbolic groups and $M\subseteq H_1$ be a finite subset.
Then $G$ has a non-elementary torsion-free Gromov hyperbolic  quotient $\overline{G}$  such that the canonical projection $G \twoheadrightarrow \overline{G}$ is surjective on $H_2$ and injective on $M$. 
 \end{thm}

This result, together with our main Theorem~\ref{main},  indeed yields Corollary \ref{C: T}.  
Take for $H_1$ our torsion-free Gromov hyperbolic group without the unique product property for the sets $A$ and $B$ produced by Theorem~\ref{main}.
Take for $H_2$ a Gromov hyperbolic group with Property (T) (e.g. a discrete subgroup of finite covolume in $Sp(n, 1)$) and for $M$ a finite subset of $H_1$ containing $A$, $B$, and $AB$.
By Theorem~\ref{olsh}, we get a torsion-free Gromov hyperbolic group $\overline{G}$ with Property (T) and without the unique product property.\smallskip 

Our two ways  to construct groups in Corollary \ref{C: T} have distinct outcomes: the first approach shows the existence of graphical small cancellation
presentations of such groups and the preceding argument provides first explicit examples.

A further strong consequence of our results is the existence of Tarski monster groups without the unique product property.

\begin{cor}\label{cor:allcyclic}
 There are torsion-free Property (T) groups $G$ without the unique product property such that all proper subgroups of $G$ are cyclic. 
\end{cor}

Indeed, it follows from~\cite{olshanskii_residualing_1993}*{Th. 2}  that every non-cyclic torsion-free Gromov hyperbolic group $G$ has a non-abelian torsion-free quotient $\widetilde{G}$ such that all proper subgroups of $\widetilde{G}$ are cyclic, and that $G\twoheadrightarrow \widetilde{G}$ is injective on any given finite subset of $G$ \cite{olshanskii_residualing_1993}*{ Cor. 1}. Applied to a finite subset containing $A$, $B$, and $AB$ in a group $G$ given by Theorem~\ref{main}, this yields Tarski monster groups without the unique product property.

In particular, we obtain the first examples of groups without the unique product property all of whose proper subgroups are unique product groups.   
Again, explicit recursive presentations are available for such new monster groups.

Our constructions are of particular interest also in the context of the following two important open problems.

\begin{pbm}
 Do the Rips-Segev groups without the unique product property satisfy the Kaplansky zero-divisor conjecture?
\end{pbm}

 Combining recent deep results~\cites{schreve_strong_2013,linnell_strong_2012,agol_virtual_2013}, we observe that the Kaplansky zero-divisor conjecture holds for all torsion-free CAT(0)-cubical\footnote{A group is \emph{CAT(0)-cubical} if it admits  a proper cocompact action on a CAT(0)-cubical complex.} Gromov hyperbolic groups over the field of complex numbers. Our groups from Corollary~\ref{C: T} are not CAT(0)-cubical as they are infinite Property (T) groups.
Thus, it follows from our results that the CAT(0)-cubulation cannot solve the Kaplansky zero-divisor conjecture for all Gromov hyperbolic groups without the unique product property.

 \begin{pbm}\label{Q: res fin}
  Is every Gromov hyperbolic group residually finite? 
 \end{pbm}

 If $Q$ is finite then $N$ in our construction is normal of finite index and without the unique product property. 
Every residually finite Gromov hyperbolic group has a finite index subgroup with the unique product property by a result of Delzant~\cite{delzant_sur_1997}.
Then the following questions arise naturally. 
\begin{itemize}
\item Does there exist a Gromov hyperbolic group all of whose normal finite index subgroups are without the unique product property?
\item Does there exist a Gromov hyperbolic group all of whose subgroups of index at most $ k$, for a given $k \geqslant 2$, 
  are without the unique product property? 
  \end{itemize}
  
The last question has recently been  answered in affirmative~\cite{GMS}, via a further application of the generalized Rips-Segev graphs. \medskip


\noindent
{\bf Acknowledgments.} 
Both authors are partially supported by the ERC grant ANALYTIC no.\  259527 of G. Arzhantseva.
The second author is a recipient of the DOC fellowship of the Austrian Academy of Sciences and was partially supported by the University
 of Vienna research grant 2013.

We thank the Erwin Schr\"odinger International Institute for Mathematical Physics in Vienna for supporting the workshop ``Geometry and Computation in Groups'' and participants A. Minasyan and F. Dahmani, who, after they learned from our results, pointed out that Ol'shanskii's small cancellation theory helps  to produce more Property (T) groups without the unique product property. 

\section{Rips construction via the free product of groups}
In this section, we review the original Rips construction~\cite{rips_subgroups_1982} but regard it in the context
of small cancellation theory over the free product of groups.  This allows us to explicit the choice of group relators in an easier way and hence, to
provide concrete group presentations of the middle group, both in the original and in our new short exact sequences of groups, see Theorems~\ref{main} and~\ref{main2}.

Let $Q=\langle x_1,\ldots,x_m \mid r_1,\ldots, r_n\rangle$ be a finitely presented group. 

Let $a,b\not\in \{x_1^{\pm 1}, \ldots, x_m^{\pm 1}\}.$ We
consider the free product $\langle x_1,\ldots, x_m, a\rangle \ast \langle b\rangle,$
equipped with the  \emph{free product length} $\vert \cdot \vert_{\ast}$, also known as the \emph{syllable-length}~\cite{lyndon_combinatorial_1977}.

Let $H$ be a group defined by a presentation $\langle x_1, x_2,\ldots, x_m, a, b \mid R\rangle,$ where the set $R$ of relators consists of the following $n+4m$ words: 
\begin{align}
 &    r_i a^{10i-9}b a^{10i-8}ba^{10i-7}b\cdots ba^{10i}b,\, 1\leqslant i \leqslant n;\\
  & x_j^{-1}ax_ja^{10(n+j)-9}ba^{10(n+j)-8}ba^{10(n+j)-7}b\cdots ba^{10(n+j)}b,\,    1\leqslant j \leqslant m,\\
& x_jax_j^{-1}a^{10(m+n+j)-9}ba^{10(m+n+j)-8}ba^{10(m+n+j)-7}b\cdots ba^{10(m+n+j)}b,\, 1\leqslant j \leqslant m; \nonumber \\
   & x_j^{-1}bx_ja^{10(2m+n+j)-9}ba^{10(2m+n+j)-8}ba^{10(2m+n+j)-7}b\cdots ba^{10(2m+n+j)}b,\,  1\leqslant j \leqslant m,\\
   & x_jbx_j^{-1}a^{10(3m+n+j)-9}ba^{10(3m+n+j)-8}ba^{10(3m+n+j)-7}b\cdots ba^{10(3m+n+j)}b,\,  1\leqslant j \leqslant m, \nonumber
  \end{align}
  and the length on $H$ is  the free product length induced from $\langle x_1,\ldots,x_m, a\rangle \ast \langle b\rangle$.

In the terminology of the \emph{small cancellation theory over the free product}~\cite{lyndon_combinatorial_1977}*{Ch. V. 9},  the \emph{pieces} in these relators have length ($=$ the free product length) at most $3$ and the relators have length $19$. Hence, this presentation of $H$ satisfies the classical  free product $C'(1/6)$--small cancellation condition over  $\langle x_1,\ldots,x_m,a\rangle * \langle b\rangle$. 

It follows that $H$ is torsion-free~\cite{lyndon_combinatorial_1977}*{Th. 10.1, Ch. V} and Gromov hyperbolic~\cite{pankratev_hyperbolic_1999}. Let $N$ be the subgroup generated by $a$ and $b$. The relators (2) and (3) guarantee that $N$ is normal in $H$. Thus, $N$ coincides with the kernel of the epimorphism $H\twoheadrightarrow Q$ which maps $a\mapsto 1$, $b\mapsto 1$ and $x_i\mapsto x_i$ for all $i$. We conclude the following significant result of Rips.
\begin{prop}\label{rips} \cite{rips_subgroups_1982}
 Let $Q$ be  a finitely presented group. Then there exists a short exact  sequence \[1\to N \hookrightarrow H \twoheadrightarrow Q \to 1\] such that 
 \begin{itemize}
  \item $H$ is a torsion-free Gromov hyperbolic group, 
  \item $N$ is a non-trivial normal $2$-generator subgroup of $H$. 
  \end{itemize}
  \end{prop}

\section{Rips construction without unique product}\label{sec:graphical}
We define here our group $G$ required by Theorem~\ref{main}. 
We begin with definitions.

 Let $\mathcal{G}$ be a (finite or infinite) graph with directed edges. A labeling $\ell$ of $\mathcal{G}$ by $\{x_1,\ldots, x_n,a,b\}$ assigns to every edge $e$ a letter $y$ equal to $x_i$, $a$ or $b$, so that going along $e$ in positive direction we read $y$ and going along $e$ in negative direction we read $y^{-1}$.

 A path in $\mathcal{G}$ is \emph{reduced} if 
it has no backtracking. Every reduced labeled path in  $\mathcal{G}$ bears as label a word in letters from $\{x_1,\ldots, x_n, a, b\}$. 
 On the other hand, for every such a word $w$ there is a reduced path labeled by $w$.
 
 Let $e_1$ and $e_2$ be edges in $\mathcal{G}$ with a common vertex $v$ which are either both directed towards $v$ or both directed away from $v$, and such that $\ell(e_1)=\ell(e_2)$. A \emph{Stallings folding} (briefly, a folding) of $\mathcal{G}$ is the identification of two such edges.  A labeling of $\mathcal{G}$ is \emph{reduced} if it does not admit any foldings. 
 
A \emph{(graphical) piece} in a reduced labeled graph $\mathcal{G}$ is a reduced labeled path which has at least two distinct immersions into $\mathcal{G}$.

The labeling satisfies the \emph{$Gr'(1/6)$--graphical small cancellation condition} whenever for all pieces $p$ in $\mathcal{G}$ we have 
  \[\vert \ell(p)\vert< \frac16 \min\{|\ell(c)|  \mid \text{ $c$  is a non-trivial cycle in  } \mathcal{G}\},\] where $| \cdot |$ denotes the usual word length metric on the free group on the free generating set $\{x_1,\ldots, x_n, a, b\}$.
   
   The labeling satisfies the \emph{$Gr_{\ast}'(1/6)$--graphical small cancellation condition} whenever for all pieces $p$ in $\mathcal{G}$ we have
  \[\vert \ell(p)\vert_{\ast}< \frac16 \min\{|\ell(c)|_{\ast}  \mid \text{ $c$ is a non-trivial cycle in  } \mathcal{G}\}.\]

 A \emph{generalized Rips-Segev graph} $\mathcal{RS}$, associated to given non-empty finite subsets $A$ and $B$ of elements in $\langle a\rangle * \langle b\rangle$,
is a connected finite reduced graph labeled by $\{a, b\},$ whose labeling satisfies the $Gr_*(1/6)$--graphical small cancellation condition, and whose cycles are labeled by words expressing the non unique product property of
 $AB$   in a group generated by $a,b$ subject to these relator words~\cite{markus}.  We explain the construction of such graphs on a concrete example, which we use then to prove our theorems.
 
  Let $c_i:=b^{i}a^{10^{J+10i-9}}b\cdots b a^{10^{J+10i-5}}$ for $i\in \{1,\ldots, K\}$ and some integers $K,J\geqslant 1$ to be specified below.  
  Let 
 $$A:=\bigsqcup_{i=1}^K\left\{c_i,c_ia,c_ia^2,\ldots,c_ia^{10^{J+10i-1}}\right\} \hbox{ and } B:=\{1,a,b,ab\}.$$ 
  We now produce a graph encoding the non unique product property of $AB.$
  
   We first choose a finite connected regular covering graph $\Theta$ of the (oriented) bouquet of 4 cycles labeled by
   the letters $y_1,y_2,y_3,y_4$ such that the length of non-trivial cycles in $\Theta$ is at least $19$. Such a covering graph does exist as the free group  (= the fundamental group of the bouquet) is residually finite. Indeed, for every ball $B(r)$  of radius $r$ in the Cayley graph of the free group there exists a finite index normal subgroup $N$ such that $B(r)\cap N$ is the identity. The covering graph corresponding to $N$ is a finite connected graph with vertex degree $8$ and girth at least $2r$. Taking $r\geqslant 10$ yields the required covering graph $\Theta.$  
   
   Let us enumerate the vertices of $\Theta$ by $1, \ldots, i, \ldots K$. As it is the covering graph of the bouquet labeled by $y_1,y_2,y_3,y_4$, for all $1\leqslant j \leqslant 4$ each vertex $i$ has an edge $y_{ij}:=( l_{ij},i)$ and an edge $z_{ij}:=( i,k_{ij})$  such that $z_{ij}=y_{k_{ij}j}.$ These are the ``in'' and ``out'' edges at vertex $i$ labeled by $y_j$ if we consider the labeling of $\Theta$ induced from the bouquet by the covering map.
   
   We define a new labeling $L$ of $\Theta$ as follows. For each $i$, we set 
    \begin{align*}
   L(y_{i1}):=&ba^{-10^{J+10i-4}}\\
   L(y_{i2}):=&a^{10^{J+10l_{i2}}}ba^{-10^{J+10i-3}}\\
   L(z_{i3}):=&a^{10^{J+10i-2}}ba^{-10^{J+10k_{i3}}}\\
   L(z_{i4}):=&a^{10^{J+10i-1}}b .
   \end{align*}
   The graph $\Theta$ coincides with the graph used by Rips-Segev in their original construction \cite{rips_torsion-free_1987}. 
   
   Now we subdivide and label edges of $\Theta$ according to $L$, so that an  edge bears a letter from $\{a, b\}$ and has the respective orientation.
   For instance, the edge $y_{i1}$ of $\Theta$ becomes a path of length $10^{J+10i-4}+1$ all of whose vertices but the endpoints are of degree 2 and whose label is $ba^{-10^{J+10i-4}}$. 
    The resulting labeled graph, denoted by $\Theta'$, is not reduced.  For our purposes, we further reduce  the graph (that is, make all possible foldings).
  Let us denote such a reduced graph by $\Theta''$. Note that certain  vertices of $\Theta''$ are identified with vertices $1, \ldots, i, \ldots, K$ of $\Theta$: the vertices of $\Theta$ we started with have not been  identified while producing $\Theta''$.  
   
 Observe that a non-zero exponent $P_i$ in words $$a^{Q}b^{\varepsilon_1}a^{P_i}b^{\varepsilon_2}a^{Q'} \hbox{ with }  \varepsilon_t=\pm 1,$$ read on reduced paths  in $\Theta''$ (hence the reduction of the label on a path $({l_{ij}},i),(i,{k_{is}})$ in $\Theta'$) is unique among the numbers   
\[
10^{J+10i-4}9,\ldots
,10^{J+10i-4}9999,10^{J+10i-3}9, \ldots,
10^{J+10i-1}9,10^{J+10i-4},\ldots, 10^{J+10i}, \] that is,  all $P_i$'s are pairwise distinct.

 We now use $\Theta''$ to produce a new graph. Namely,
for  each $i$ we glue a path $\gamma_i$ labeled by $c_i$ to $\Theta''$ so that
vertex $i$ of $\Theta''$ is identified with the endpoint of $\gamma_i$.
The starting points of all $\gamma_i$'s are all identified with a new vertex, denoted by $0$, and 
$\gamma_i$ has no vertices other than its endpoint in common with $\Theta''$.
Let us fold this new graph and denote the result  by $\mathcal{RS}$. 
We see, as above, that the exponents $P$ of $a$ in paths labeled by $b^{\varepsilon_1}a^Pb^{\varepsilon_2}$ are pairwise distinct.


 The reduced graph $\mathcal{RS}$ is our generalized Rips-Segev graph.  
In particular, the graphical pieces in $\mathcal{RS}$ consist of paths $a^{P}$, $a^{P}b^{\varepsilon}a^{P'},$ and $a^{P}b^{2\varepsilon}a^{P'}$, where $P$ and $P'$ are possibly zero. Therefore, the free product length of the pieces in $\mathcal{RS}$ is at most $3$. The free product length of cycles in $\mathcal{RS}$  is at least $19$. It follows that the labeling of $\mathcal{RS}$ satisfies the $Gr_*'(1/6)$--graphical small cancellation condition over the free product $\langle x_1,\ldots,x_m,a\rangle*\langle b\rangle$. 

If we choose $J$ such that $9(n+4m)<10^{J}$, then $P_i$ is not among the numbers $1,\ldots, 9(n+4m)$ for all $i$. 
Let $\mathcal{G} $ be the disjoint union of $\mathcal{RS}$ and of the $(n+4m)$ cycles labeled by words (1), (2), and (3) defined above. The graphical pieces in $\mathcal{G}$ are of length at most $3$ and cycles are 
labeled by words whose free product length is at least 19 in $\langle x_1,\ldots,x_m,a\rangle*\langle b\rangle$. Therefore, the labeling of $\mathcal{G}$ satisfies the $Gr_*'(1/6)$--graphical small cancellation condition over the free product $\langle x_1,\ldots,x_m,a\rangle*\langle b\rangle$.  

Let $G$ be a group generated by $x_1,\ldots, x_m, a, b,$  subject to relators defined by $\mathcal{G}$, 
$$G:= \langle x_1,\ldots, x_m, a, b \mid\hbox{labels on reduced cycles of $\mathcal{G}$}\rangle,$$ also denoted briefly $ G=\langle x_1,\ldots, x_m, a, b \mid \mathcal{G}\rangle$. The following general result ensures that $G$ is torsion-free and Gromov hyperbolic.

\begin{thm}\cite{markus}*{Th. 1}\label{T: markus}
 Let $G_1,\ldots,G_n$ be finitely generated groups.  Let $\mathcal{G}$ be a family of finite connected graphs edge-labeled by $G_1\cup \ldots \cup G_n$ so that the $Gr_*'(1/6)$--graphical small cancellation condition with respect to the free product length on the free product $G_1*\cdots *G_n$ is satisfied.
Let $G$ be the group given by the corresponding graphical presentation, that is,  the quotient of $G_1*\cdots *G_n$ subject to the relators being  the words read on the cycles of $\mathcal{G}.$ 

Then
$G$ satisfies a linear isoperimetric inequality with respect to the free product length. Moreover, $G$ is torsion-free whenever $G_1,\ldots,G_n$ are torsion-free; 
$G$ is Gromov hyperbolic whenever $G_1,\ldots,G_n$ are Gromov hyperbolic and $\mathcal{G}$ is finite.
\end{thm}

Thus, it suffices to check that $G$ does not have the unique product property.  
This, together with Proposition \ref{rips}, will then imply  our main result, Theorem~\ref{main}.

\begin{thm}\label{T: up} The group $G$ does not satisfy the unique product property. 
Namely, the sets $A$ and $B$ embed and do not have the unique product property  in $G$.\end{thm}

\begin{proof}
 The set $R_G$ of relators of $G$ is the set of all words read on the reduced cycles of $\mathcal{G}$. By definition, $R_G$ is the disjoint union of the above defined set $R$ 
 and the words read on the cycles of $\mathcal{RS}$. The latter words encode the non unique product property  of finite sets  $A$ and $B$ of cyclically reduced words in the letters 
 $a^{\pm 1}$ and $b^{\pm 1}$ (in the group $\langle a, b \mid \mathcal{RS}\rangle$ defined by $\mathcal{RS}$). Denote by $A'$ and $B'$ the image of $A$ and $B$ in $G$.  Let us show that the maps $\iota_A:A\to A'$ and $\iota_B:B\to B'$ are injective. By the construction of $\mathcal{RS}$ it will follow that  $A$ and $B$ do not have the unique product property in $G$.

 We use the language of van Kampen diagrams over the free product~\cite{lyndon_combinatorial_1977}.
 
 Firstly, let us fill each of the reduced cycles in $\mathcal{G}$ with a disc and denote  by $\widetilde{\mathcal{G}}$ the resulting 2-complex. Let $D$ be a van Kampen diagram over $R_G$. Let $\Pi$ be a  face of $D$ with a boundary word $r\in R_G$. The face $\Pi$ is a copy of the 2-cell in $\widetilde{\mathcal{G}}$ with the same boundary word $r$. We say, $\Pi$ \emph{lifts} to $\widetilde{\mathcal{G}}$. The edges in the boundary of $\Pi$ lift to $\widetilde{\mathcal{G}}$ with $\Pi$. The $Gr_*'(1/6)$--graphical small cancellation condition immediately implies that such a lift of $\Pi$ to $\widetilde{\mathcal{G}}$ is unique. An edge common to faces $\Pi_1$ and $\Pi_2$ in $D$ \emph{originates} in $\mathcal{G}$ if its lifts to $\widetilde{\mathcal{G}}$ with $\Pi_1$ and with $\Pi_2$ coincide.
  An edge in $D$ is called \emph{exterior} if it is contained in the boundary of at most one face, otherwise it is called \textit{inner}. A face which has an exterior edge in its boundary is an  \emph{exterior face}. The collection of exterior edges is the \emph{boundary} $\partial D$ of $D$. An \emph{arc} is a path of whose vertices only the starting and terminal vertex can have vertex degree $>2$.  An \emph{inner segment} is an arc of  inner edges which are all common to  the same pair of faces. 

 A priori, a van Kampen diagram over $R_G$ does not satisfy the $C'(1/6)$--small cancellation condition over the free product as the inner segments whose edges originate in $\mathcal{G}$ do not correspond to graphical pieces. 
 To remedy this, we assume that $D$ is \emph{minimal}, that is, it has the minimal number of faces among all van Kampen diagrams over $R_G$ with the given boundary label.
 Then such a minimal $D$  does satisfy the $C'(1/6)$--small cancellation condition over the free product as it has no inner edges that originate in $\mathcal{G}$.  The van Kampen diagrams with more than one face which satisfy the $C'(1/6)$--small cancellation condition over the free product have at least two exterior faces (cf.~\cite{markus}*{Th. 1.11}). The free product length of the boundary of such a van Kampen diagram is at least the free product length of the longest boundary of its faces.
 
For each $\overline{a}\in A$ the generalized Rips-Segev graph $\mathcal{RS}$ contains a reduced path labeled by the word  $\overline{a}$, whose  starting vertex is $0$ and whose 
 terminal vertex is uniquely determined by the label $\overline{a}$ as $\mathcal{RS}$ is reduced.  
 Suppose that $a_1\not=a_2 \in A$ and $a_1=_Ga_2,$ where the equality is in $G$. Let $p$ be a reduced path  in $\mathcal{RS}$ connecting the vertices $a_1$ and  $a_2$ (that is, connecting the endpoints
 of two paths starting at vertex $0$ and labeled by $a_1$ and $a_2$, respectively). Let $x$ be the label 
 of $p$. As $x=_G1$ by assumption, there is a non-empty minimal van Kampen diagram $D$ over $R_G$ with  boundary label $x$. Then $D$ satisfies the $C'(1/6)$--small cancellation condition 
 over $\langle x_1,\ldots,x_m,a\rangle*\langle b\rangle$.  
 We choose such a path $p$ so that the number of faces of $D$ is minimal among the numbers of faces of all possible minimal van Kampen diagrams over $R_G$ with  boundary label
 corresponding to all possible reduced paths connecting $a_1$ and $a_2$ in $\mathcal{G}$. 
  
    A face $\Pi$ in such a $D$ is \emph{originating with $\partial D$} if the lift of the exterior boundary of $\Pi$ to $\widetilde{\mathcal{G}}$ with $\Pi$ coincides with the (obvious) lift of 
    the exterior boundary of $\Pi$ to $\widetilde{\mathcal{G}}$ with $\partial D$ (which is the lift of $p$ from $\mathcal{G}$ to $\widetilde{\mathcal{G}}$ induced by the trivial inclusion $\mathcal{G}\subseteq \widetilde{\mathcal{G}}$). We claim that $D$ has no face originating with $\partial D$. Indeed, otherwise we can remove such a face from $D$. Namely, we delete the interior of such a face together with the interior of its exterior boundary which lifts to $\widetilde{\mathcal{G}}$ with $\partial D$. Then we delete all faces which do not belong to the component containing the vertex representing  $a_1$ and $a_2,$ respectively. 
    The resulting van Kampen diagram has a fewer number of faces and
    the lift of its boundary to $\widetilde{\mathcal{G}}$ is a path which connects $a_1$ and $a_2$. This contradicts the minimality condition in the choice of $p$.
  
 The path $p$ lies on a cycle $c=pp'$ in $\mathcal{RS}$. Let $x'$ be the label of $p'$ and $w=xx'$ be the label of $c$, that is, $w=_G1$.
 Let $D'$ be the minimal van Kampen diagram over $R_G$ with  boundary label $w$ which globally (that is, with all of its faces) lifts to $\widetilde{\mathcal{G}}$. Obviously, $D'$ consists of the only face whose boundary label is $w$. Let us glue $D'$ to $D$ along $\partial D$ labeled by $x$. We denote the resulting diagram by $D''$.
 Then $D'$ is the only exterior face of $D''$. We have just seen that $D$ has no faces originating with $\partial D$. Therefore, the inner segments of $D''$ which lie on $\partial D$ are all graphical pieces. Hence, $D''$ satisfies the $C'(1/6)$--small cancellation condition over $\langle x_1,\ldots,x_m,a\rangle*\langle b\rangle$. This leads to a contradiction as $D''$ has to contain at least two exterior faces. Thus, there is no distinct elements $a_1,a_2\in A$ such that $a_1=_Ga_2$ in $G$.
 We conclude that $\iota_A:A\to A'$ is injective.

Now $B=\{1,a,b,ab\}$. If $b_1\not=b_2 \in B$ and $b_1=_Gb_2$, then $|b_1b_2^{-1}|_*<3$. This is a contradiction as the graphical small cancellation condition over the free product $\langle x_1,\ldots,x_m,a\rangle*\langle b\rangle$  implies that the free product length of relations in $G$ has to be at least $19$. Therefore, $\iota_B:B\to B'$ is injective as well. Thus,
$G$ does not satisfy the unique product property.
\end{proof}

The groups $G$ we have just constructed are the only known examples of torsion-free Gromov hyperbolic groups without the unique product property.
It is yet unknown whether these torsion-free groups without the unique product property do satisfy the Kaplansky zero-divisor conjecture.

The proof of Theorem~\ref{main2} is straightforward by the arguments above applied to an infinite generalized Rips-Segev graph, that is,  to an infinite disjoint union of finite Rips-Segev graphs. 
Several explicit constructions of such infinite families of finite graphs can be found in~\cite{markus}.  
A finite subunion of such a family yields a torsion-free Gromov hyperbolic group without the unique product property as above,  
whence the direct limit of such groups for the resulting group $G$ in Theorem~\ref{main2}.
%
\section{Examples of torsion-free groups without unique product}\label{sec:algo}

We now vary the quotient group $Q$ in our new short exact sequence above, Theorem~\ref{main}. This is to provide a wide diversity of Gromov hyperbolic torsion-free groups $G$ without the unique product property which satisfy many unusual algebraic, geometric, and algorithmic properties. 

Recall that the generalized word problem, also known as the membership problem in a group $G$, asks for an algorithm to decide whether a given word in the generators of $G$ represents an element 
of a given finitely generated subgroup of $G$. The following statements are immediate generalizations of \cite{rips_subgroups_1982}.

\begin{prop}
The generalized word problem is unsolvable in the class of torsion-free Gromov hyperbolic groups without the unique product property.
\end{prop}
\begin{proof} Let $Q$ be a finitely presented group with unsolvable word problem. Explicit finite presentations of such groups can be found in  \cites{boone_word_59,borisov_simple_1969, collins_simple_1986}. As pointed out by \cite{baumslag_unsolvable_1994}, the membership problem for $N$ in $G$ (a new group, produced by our Theorem~\ref{main}) is equivalent to the word problem in $Q$. 
Indeed, a word $w$ in the generators of $G$ represents an element of $N$ if and only if $w=_Q1$, where the equality is in $Q$. 
\end{proof}

\begin{prop} There exists a torsion-free Gromov hyperbolic group $G$ without the unique product property such that
\begin{itemize}
\item there are finitely generated subgroups $P_1,P_2$ of $G$ such that $P_1\cap P_2$ is not finitely generated;
\item there is a finitely generated but not finitely presented subgroup of $G$;
\item for any $r\geqslant 3$ there is an infinite strictly increasing sequence of $r$--generated subgroups of $G$.
\end{itemize}
\end{prop}

More algorithmic properties in the context of Rips construction are studied by \cite{baumslag_unsolvable_1994}. Applied to our situation they yield the following.
\begin{prop}\label{P: alg prop} There exists a torsion-free Gromov hyperbolic group $G$ without the unique product property such that there is no algorithm to determine
\begin{itemize}
\item  the rank of $G$;
\item  whether an arbitrary finitely generated subgroup of $G$ has finite index;
\item whether an arbitrary finitely generated subgroup of $G$ is normal;
\item  whether an arbitrary finitely generated subgroup of $G$ is finitely presented;
\item  whether an arbitrary finitely generated subgroup $S$ of $G$ has a finitely generated second integral homology group $H_2(S,\mathbb{Z})$.
\end{itemize}
\end{prop}

The proofs are by choosing a group $Q$ with the required property, which then allows to pullback the property to the group $H$ and then (immediately, for the above algorithmic properties) to $G$. The presentations of $Q$ and, hence, of $H$ and of $G$ can be given explicitly as in the previous section. 
 %
 \section{(T)-Rips-construction without unique product}
 
 Our generalization of the Rips construction allows to combine the original, now classical, arguments with further geometric properties by adding suitable new relators to the presentation of $H.$ 
 In our Theorem~\ref{main}, these new relators encode the non unique product property and are given by the generalized Rips-Segev graph. 
Another famous property that can be encoded by graphs is Kazhdan's Property (T). For instance, Gromov's spectral characterization of Property (T) 
  allows him to produce finitely presented groups with Kazhdan's Property (T) given by the graphical $Gr'(1/6)$--small cancellation presentations~\cites{gromov_random_2003, silberman_random_2003}. Mixing the original Rips construction and the above Gromov result, Ollivier and Wise~\cite{ollivier_kazhdan_2007} obtain a short exact sequence
  $1\to N \hookrightarrow G \twoheadrightarrow Q \to 1$, where $G$ is a torsion-free group defined by 
 a finite graphical $Gr'(1/6)$--small cancellation presentation and $N$ has Kazhdan's Property (T). We extend their result as follows.

 \begin{thm}\label{T: T-rips} Given a finitely presented group $Q$, there is a short exact sequence \[1\to N \hookrightarrow G \twoheadrightarrow Q \to 1\] such that
 \begin{itemize}
 \item $G$ is a torsion-free non-elementary Gromov hyperbolic group,
 \item $N$ has Kazhdan's Property (T) and does not satisfy the unique product property.
 \end{itemize}
 \end{thm}
\begin{proof}
Let $\langle x_1,\ldots x_m\mid r_1,\ldots, r_n\rangle$ be a presentation of $Q$.
 Let $a_1,\ldots,a_{35},b$ be distinct and different from  each of $x_1,\ldots, x_m$. Let $\mathcal{T}$ be the finite graph provided by our Theorem~\ref{T: ap1} below, with a labeling by $\{a_1, \ldots, a_{35},b\}$ such that the group $\langle a_1,\ldots,a_{35},b \mid  \mathcal{T} \rangle$ defined by this graph satisfies Property (T) and the labeling of $\mathcal{T}$ satisfies the $Gr_*'(1/6)$--graphical small cancellation condition. 
  Let $M$ exceed the largest exponent of $a_1$ in the labeling of $\mathcal{T}$.  We
take the following new explicit Rips relators, see our version of the Rips construction in Section 2. 
\begin{align}
 \label{TRel1}   & r_ia_1^{10i-9+M} b a_1^{10i-8+M}b\cdots ba_1^{10i+M}b \text{ for all } 1\leqslant i \leqslant n,
 \end{align}
 \begin{align}
   \label{TRel2}
   & x_j^{-1}a_kx_j^{}a_1^{10(n+(k-1)m+j)-9+M}ba_1^{10(n+(k-1)m+j)-8+M}b\cdots \\
  &\qquad \qquad \qquad \qquad \cdots ba_1^{10(n+(k-1)m+j)+M}b \text{, for all } 1\leqslant j \leqslant m,1\leqslant k\leqslant35, \nonumber\\
   & x_ja_kx_j^{-1}a_1^{10(n+(k+34)m+j)-9+M}ba_1^{10(n+(k+34)m+j)-8+M}b\cdots \nonumber\\
    & \qquad \qquad \qquad \quad \quad  \cdots ba_1^{10(n+(k+34)m+j)+M}b \text{, for all } 1\leqslant j \leqslant m,1\leqslant k\leqslant35,  \nonumber 
    \end{align}
    \begin{align}
   \label{TRel3} 
   & x_j^{-1}bx_ja_1^{10(71m+n+j)-9+M}ba_1^{10(71m+n+j)-8+M}b \cdots ba_1^{10(71m+n+j)+M}b
   \\& \qquad \qquad \qquad \qquad \qquad \qquad \qquad \qquad \qquad \qquad \quad  \text{ for all } 1\leqslant j \leqslant m, \nonumber \\
   &  x_jbx_j^{-1}a_1^{10(72m+n+j)-9+M}ba_1^{10(72m+n+j)-8+M}b\cdots ba_1^{10(72m+n+j)+M}b\nonumber \\
   &\qquad \qquad \qquad \qquad \qquad \qquad \qquad \qquad \qquad \qquad \quad \text{ for all } 1\leqslant j \leqslant m.\nonumber
  \end{align} 
 
 Let $\mathcal{R}$ be the disjoint union of $72m+n$ cycles, each labeled by one of these Rips relators.  
 Let $M':=10(72m+n)+M$. Take a generalized Rips-Segev graph $\mathcal{RS}$ for $a:=a_1$ and $b$,  where $J$ is chosen such that $M'<10^{J}$. 

Our new group $G$ is defined by the following presentation, $$G:=\langle x_1,\ldots,x_m,a_1,\ldots,a_{35},b\mid  \mathcal{T} \sqcup \mathcal{R} \sqcup \mathcal{RS}\rangle,$$ that is, the relators of $G$ are the labels of the reduced cycles of $\mathcal{T} \sqcup \mathcal{R} \sqcup \mathcal{RS}$. Let $N$ be the subgroup of $G$ generated by $a_1,\ldots,a_{35},b$. This subgroup $N$ is normal in $G$ by our Rips relators \eqref{TRel2} and \eqref{TRel3} read on $\mathcal{R}$. The map $G\to Q$, defined by $x_i\mapsto x_i$ and $a_i,b\mapsto 1$, is an epimorphism. The kernel of this map is generated by $a_1,\ldots ,a_{35},b$ and therefore coincides with $N$. 
  
 The labeling of $\mathcal{T} \sqcup \mathcal{R} \sqcup \mathcal{RS}$ satisfies the $Gr'_*(\frac16)$--small cancellation condition over $\langle x_1,\ldots,x_m, a_1\rangle*\langle a_2\rangle* \cdots * \langle a_{35}\rangle * \langle b\rangle,$ that is, the graphical small cancellation conditions with respect to the free product length $\vert \cdot \vert_*$  in  $\langle x_1,\ldots,x_m, a_1\rangle*\langle a_2\rangle* \cdots * \langle a_{35}\rangle * \langle b\rangle.$  Indeed, the reduced non-trivial cycles in $\mathcal{T}$, $\mathcal{R}$, and $\mathcal{RS}$ have free product length at least $20$. The immersed subpaths common in $\mathcal{T}$, $\mathcal{R}$ and $\mathcal{RS}$ are of free product length at most $3$, by our choice of the $a_1$-exponents. 

 Theorem \ref{T: markus} implies that $G$ is torsion-free and Gromov hyperbolic, and the proof of Theorem \ref{T: up}, applied to the graph $\mathcal{T} \sqcup \mathcal{R} \sqcup \mathcal{RS}$, 
 shows that $G$ is without the unique product property. 
 
 As a subgroup, the group $N$ is injected in $G$. Therefore, given two words $w_1$ and $w_2$ in $a_1,\ldots,a_{35},b$ such that $w_1\not =_{N} w_2$ in $N$, we have that $w_1\not =_{G} w_2$ in $G$. This implies that 
\[
N=\langle a_1,\ldots,a_{35},b\mid \mathcal{T}\sqcup \mathcal{RS}\sqcup \hbox{ relations in $G$ in letters $a_1^{\pm1},\ldots,a_{35}^{\pm1},b^{\pm1}$}\rangle.
  \]
 Thus, $N$ is  without the unique product property. Indeed, the sets $A$ and $B$ defining $\mathcal{RS}$ are contained in $N$ and the relations read on $\mathcal{RS}$ imply that $N$ does not have the unique product property for $A$ and $B$. 
 The group $N$ is  a quotient of the group $\langle a_1,\ldots, a_{35},b\mid \mathcal{T}\rangle$. This group has Property (T)  by Theorem \ref{T: ap1}.  Thus, $N$ has Kazhdan's Property (T) as well.
\end{proof}

Observe that our use of the free product language above simplifies the arguments of \cite{ollivier_kazhdan_2007}*{Prop. 2.2} and the proof of \cite{ollivier_kazhdan_2007}*{Th. 1.1}  which justify that relators in the Rips construction can be added to the relators defined by the graph $\mathcal{T}$.

\appendix

\section{Small cancellation labellings and Property (T)}\label{sec:ap}

Given a graph $\mathcal{G}$ labeled by $\{a_1,\ldots,a_m\}$, we denote by $G(\mathcal{G})$ the group defined by $\langle a_1,\ldots,a_m\mid \mathcal{G}\rangle.$ 
Our aim is to prove the following result.
   
 \begin{them}\label{T: ap1} For all $m> 35$, there exists a  finite  connected graph $\mathcal{T}$ labeled by $\{a_1,\ldots,a_m\}$ such that the labeling satisfies the free product $Gr_*'(\frac16)$--small cancellation condition over $\langle a_1 \rangle * \ldots * \langle a_m \rangle$ and such that $G(\mathcal{G})$ has Property (T).
 \end{them}

 This generalizes the following result of Gromov \cite{gromov_random_2003}*{1.2.A, 4.8.(3)}, see also \cite{silberman_random_2003} and \cite{ollivier_kazhdan_2007}.
 
  \begin{them}[\cite{ollivier_kazhdan_2007}*{Prop. 7.1}]\label{T: ap2}
If $m\geqslant 2$, there exists a  finite  connected graph $\mathcal{T}$ labeled by $\{a_1,\ldots,a_m\}$ such that the labeling satisfies the $Gr'(\frac16)$--small cancellation condition with respect to the word length on the free group on  $\{a_1,\ldots,a_m\}$ and the group $G(\mathcal{G})$ has Property (T).
 \end{them}
 
 \label{S: explanation T-construction}
 
Our proof of Theorem \ref{T: ap1}  proceeds as the proof of Theorem \ref{T: ap2} of \cite{ollivier_kazhdan_2007}*{Sec. 7} up to appropriate technical adjustments. 
Moreover, we show that the graph $\mathcal{T}$ satisfies the conclusions of both Theorem \ref{T: ap1} and Theorem \ref{T: ap2}:
 \begin{them}\label{T: ap3} For all $m> 64$, there exists a  finite  connected graph $\mathcal{T}$ labeled by $\{a_1,\ldots,a_m\}$ such that the labeling satisfies the free product $Gr_*'(1/6)$--small cancellation condition over $\langle a_1 \rangle * \ldots * \langle a_m \rangle$, the labeling satisfies the $Gr'(1/6)$--small cancellation condition with respect to the word length metric, and  the group $G(\mathcal{T})$ has Property (T).
 \end{them}

It is not surprising that $\mathcal{T}$ satisfies the conclusions of both Theorem \ref{T: ap1} and Theorem \ref{T: ap2}, for a large enough $m$. The intuition is that the free product length in $\langle a_1\rangle*\ldots *\langle a_m\rangle$ approximates the word length in the free group on $\{a_1,\ldots,a_m\}$ as $m\to \infty$. 
Indeed, the minimal cycle length in the free product length bounds the length of the minimal cycles in the word length from below. Pieces  are words of finite length chosen 
uniformly at random. Let us evaluate the probability that the word length and the free product length of such a random word in  letters $a_1^{\pm 1}, \ldots, a_m^{\pm 1}$ coincide. Such a word is of word length equal to $n$ if it is $a_{i_1}^{P_1}a_{i_2}^{P_2}\ldots a_{i_j}^{P_{j}}$ with all coefficients $P_i\not=0$, $a_{i_j}\not=a_{i_{j+1}}$, and $\sum_{i=1}^{j}P_i=n$. Its free product length equals to $n$ if, in addition, all exponents $P_i=\pm1$. The probability that all $P_i=\pm1$ in such a word is given by $\left(\frac{2m-2}{2m}\right)^{n-1},$ which tends to $1$ exponentially as $m\to \infty$.

We provide an explicit value $m=64$, for which the approximation of the word length by the free product length is sufficient to conclude Theorem \ref{T: ap3}.

 \subsection{Ollivier-Wise's proof of Theorem \ref{T: ap2}}
 In this subsection, first we explain Ollivier-Wise's proof of Theorem \ref{T: ap2}.
 Then we extend it to our general free product setting. 
 
Given an expander graph, we endow it with a labeling chosen uniformly at random and extract  from \cite{silberman_random_2003} and \cite{ollivier_kazhdan_2007} explicit bounds on the probability that the group defined by such a labeled graph has Property (T) and the corresponding presentation satisfies the graphical small cancellation condition. 
We put an emphasis on the combination of the estimates on the occurring probabilities.
 
 Let $\mathcal{G}$ be a finite connected graph with vertex set $V(\mathcal{G})$ and a  set of undirected edges $E(\mathcal{G})$.  We denote by
 $\lambda(\mathcal{G})$ the spectral gap of $\mathcal{G}$.  The girth, denoted by $\g(\mathcal{G})$, is the minimal number of edges in a shortest non-trivial cycle of $\mathcal{G}$.

  A labeling $\ell$ of $\mathcal{G}$ by $\{x_1,\ldots, x_n,a,b\}\times \{\pm1\}$ assigns to every edge a letter $x_i$, $a$, or $b$, and an orientation. 
  We keep the notation $\mathcal{G}$ for the resulting directed graph labeled by $\{x_1,\ldots, x_n, a, b\}$.
 We say  $\mathcal{G}$ is \emph{reduced}, whenever $\mathcal{G}$ and its folding coincide.
     
  Given $m>1$, we denote by $\widetilde{\mathcal{G}}$ the graph $\mathcal{G}$ labeled uniformly at random by $\{a_1,\ldots, a_m\}\times \{\pm1\}$.  We denote the corresponding folded labeled directed graph by $\fold(\mathcal{\widetilde{G}}).$
  
  The \emph{$j$-subdivision} $\mathcal{G}^j$ of $\mathcal{G}$ is the graph $\mathcal{G}$ with every edge replaced by $j$ edges.
   Consequently, $G\left(\widetilde{\mathcal{G}^j}\right)$ denotes the group defined by the $j$-subdivision of $\mathcal{G}$ labeled uniformly at random.
  
   The probability that $G\left(\widetilde{\mathcal{G}^j}\right)$ has Property (T) is denoted by $P_T$. The probability that the map \emph{folding}$\colon\widetilde{\mathcal{G}^j}\to \fold\left(\widetilde{\mathcal{G}^j}\right)$ is a local quasi-isometric embedding is denoted by $P_{qi}$, and  the \emph{conditional} probability that the labeling of $\fold\left(\widetilde{\mathcal{G}^j}\right)$ satisfies the $Gr'(\alpha)$--small cancellation condition with respect to the word length metric, under the condition that the folding is a local quasi-isometric embedding, is  denoted by $P_{sc}$. Thus, the probability that the labeling of $\fold\left(\widetilde{\mathcal{G}^j}\right)$ satisfies the $Gr'(\alpha)$--small cancellation condition is at least $P_{qi}P_{sc}$.
   
 We extract explicit lower bounds for $P_T$, $P_{qi}$ and $P_{sc}$ from~\cites{silberman_random_2003,ollivier_kazhdan_2007}. This  allows to estimate the probability that the labeling of $\fold\left(\widetilde{\mathcal{G}^j}\right)$ satisfies the $Gr'(\alpha)$--small cancellation condition and $G\left(\widetilde{\mathcal{G}^j}\right)$ has Property (T). This probability is  at least $P_T+P_{qi}P_{sc}-1$. For certain infinite families of graphs $(\mathcal{S}_i)_i$, we  then show that the probability that $(\mathcal{S}_i^j)_i$ satisfies these properties converges to $1$ as $i\to \infty$. This provides the existence of graphs that define groups with Property (T) and whose labeling satisfies the graphical small cancellation condition.

  Let $S_l$ be the number of words of length $l$ in the letters $a_1^{\pm 1},\ldots,a_m^{\pm 1}$ that reduce to the identity in the free group on free generators $a_1,\ldots, a_m$. The \emph{gross cogrowth} of the free group  is defined by \[\eta(m):=\lim_{l\to \infty}\frac{\log_{2m}(S_{2l})}{2l}.\] The limit exists as $S_{l+l'}\geqslant S_l S_{l'}$ and, hence, $\log_{2m}(S_{2l})$ is superadditive. 
  
 The spectral radius of the simple random walk on the free group of $m$ generators equals to $(2m)^{\eta(m)-1}$. By a result of Kesten \cite{kesten_symmetric_1959}*{Th. 3}, \[(2m)^{{\eta(m)}}=2\sqrt{2m-1}.\]  The gross cogrowth satisfies  $1/2<\eta(m)<1$ and $\eta(m)\to 1/2$ as $m\to \infty$. See e.g. \cite{ollivier_sharp_2004}*{Sec. 1.2} for basic properties of the gross cogrowth.
   
We extract from \cite{silberman_random_2003}*{Cor. 2.19 p. 164} the following estimate on $P_T$. We denote by $\lambda(G)$ the smallest non-zero eigenvalue of the graph laplacian $\Delta$, while in \cite{silberman_random_2003} $\lambda(G)$ denotes the maximal eigenvalue of $1-\Delta$, cf.  \cite{ollivier_kazhdan_2007}*{comment to Prop. 7.3} and \cite{silberman_random_2003}*{Def. in Lem.2.11 p. 154 \& p.151}. (We denote the number of generators by $m$ instead of $k$ used by \cite{silberman_random_2003}.)
  \begin{pr}
 For all $m\geqslant 2$, $d\geqslant3$, $\lambda_0 >0$, and $j\in \mathbb{N}$, there exists a number 
\[g_0=g_0(m,\lambda_0,j)\]
such that if the graph $\mathcal{G}$ satisfies \begin{enumerate}
                                     \item $\g(\mathcal{G})\geqslant g_0$,
                                     \item $\lambda(G)\geqslant\lambda_0$ for all $i$,
                                     \item $3\leqslant \deg(v) \leqslant d$ for all $v\in V(\mathcal{G})$,
                                    \end{enumerate}
then $G(\widetilde{\mathcal{G}^j})$  has Property (T) with probability  
\[P_T\geqslant1-a(m,d,\lambda_0, j)e^{-b(m,d,\lambda_0,j)|V(\mathcal{G})|},\] where $a$ and $b$ are positive numbers which do not depend on $\g(\mathcal{G})$ or $|V(\mathcal{G})|$.
\end{pr}

We express the probability $P_T$ in terms of $\g(\mathcal{G})\leqslant |V(\mathcal{G})|:$  \[P_T\geqslant 1-a(m,d,\lambda_0, j)e^{-b(m,d,\lambda_0,j)\g(\mathcal{G})}.\]
The following proposition allows to compare the edge length of an immersed path $p$ in $\widetilde{\mathcal{G}^j}$ with the word length of the labeling of $p$. 

A \emph{$(c_1,c_2,c_3)$--local quasi-isometric embedding}  between metric spaces $(X,d_X)$ and $(Y,d_Y)$ is a map $f:X\to Y$ such that, whenever $d_X(a_1,x_2)\leqslant c_3$, we have 
\[c_1^{-1}d_X(a_1,x_2)-c_2\leqslant d_Y(f(a_1),f(x_2)) \leqslant c_1d_X(a_1,x_2)+c_2.\]

We use the proof of \cite{ollivier_kazhdan_2007}*{Prop. 7.8} to obtain the following.
\begin{lemma}\label{lqi} For all $m\geqslant 2$, $\beta>0$, $d\in \mathbb{N}$, $j\geqslant 1$, if $\deg(v) \leqslant d$ for all $v\in V(\mathcal{G})$,     
then the folding $\widetilde{\mathcal{G}^j}\to \fold\left(\widetilde{\mathcal{G}^j}\right)$ is a 
$\left(\frac{\eta(m)}{1-{\eta(m)}}, \beta j\g(\mathcal{G}),\g(\mathcal{G})j\right)$--local quasi-isometric embedding 
 with probability
\[P_{qi}\geqslant 1-j^2d^{\diam(\mathcal{G})+\g(\mathcal{G})}(2m)^{-(1-{\eta(m)})\beta \g(\mathcal{G})j}.\]
 \end{lemma}
 In particular, if the folding is a local quasi-isometric embedding, then it maps non-trivial cycles to non-trivial cycles.

 We extract the following estimate from \cite{ollivier_kazhdan_2007}*{Proof of Prop. 7.4, the small cancellation part}. 
\begin{lemma}\label{scwl} For all $m\geqslant 2$, $\beta>0$  such that $\frac{1-{\eta(m)}}{{\eta(m)}} - \beta>0$, $d\in \mathbb{N}$, $\alpha >0$  such that $\frac{1-\eta(m)}{2\eta(m)}-\beta>\alpha$, $j\geqslant 1$if
\begin{enumerate}
 \item  $\deg(v) \leqslant d$ for all $v\in V(\mathcal{G})$,
  \item folding  $\colon\widetilde{\mathcal{G}^j}\to \fold\left(\widetilde{\mathcal{G}^j}\right)$  is  a 
$\left(\frac{\eta(m)}{1-{\eta(m)}}, \beta \g(\mathcal{G})j,\g(\mathcal{G})j\right)$--local quasi-isometric embedding,
  \end{enumerate}   
then the labeling of $\fold\left(\widetilde{\mathcal{G}^j}\right)$ satisfies the $Gr'(\alpha)$--small cancellation condition with respect to the word length metric on the free group with probability
\[P_{sc}\geqslant 1-j^4d^{2\diam(\mathcal{G})+2\g(\mathcal{G})}(2m)^{-(1-{\eta(m)})2\g(\mathcal{G})j\alpha\left( \frac{1-{\eta(m)}}{\eta(m)} - \beta\right)}.\]
 \end{lemma}

 As $P_{sc}$ is a conditional probability, where the condition is that the folding is a local quasi-isometric embedding, we conclude:    

\begin{pr} \label{ok07}  For all $m\geqslant 2$, $\beta>0$ such that $\frac{1-{\eta(m)}}{{2\eta(m)}} - \beta>0$, $d\in \mathbb{N}$, $\alpha >0$, such that $\frac{1-\eta(m)}{2\eta(m)}-\beta>\alpha$, $j\geqslant 1$ if $\deg(v) \leqslant d$ for all $v\in V(\mathcal{G})$,    
then the labeling of  $\fold\left(\widetilde{\mathcal{G}^j}\right)$ satisfies the  $Gr'(\alpha)$--small cancellation condition with respect to the word length metric on the free group with probability at least
\[P_{sc}(m,\beta,d,\alpha,j)P_{qi}(m,\beta,d,j).\]
 \end{pr}

 \vspace{0.5cm}
  
 Let us now consider the Selberg family of graphs $\mathcal{S}:=(\mathcal{S}_i)_i$~\cite{lubotzky_discrete_1994}:
 \begin{enumerate}
  \item for all vertices $v$ in $\mathcal{S}$, $3\leqslant \deg(v) \leqslant d$ for some fixed $d\in \mathbb{N}$, 
  \item \label{as: spectral} $\lambda(\mathcal{S}_i)\geqslant\lambda_0>0$ uniformly over all $i$ for some constant $\lambda_0$,
  \item $\g(\mathcal{S}_i)_i \to \infty$ as $i\to \infty$,
  \item there is $C>1$ such that $\diam(\mathcal{S}_i)\leqslant C \g (\mathcal{S}_i)$ for all $i$.
 \end{enumerate}

Choose $\beta>0$ such that $\frac{1-{\eta(m)}}{{2\eta(m)}} - \beta>0$. For all $\alpha>0$  such that $\frac{1-\eta(m)}{2\eta(m)}-\beta>\alpha$, the probability that the labeling of $\widetilde{\mathcal{S}_i^j}$ satisfies the $Gr'(\alpha)$--small cancellation condition and $G\left(\widetilde{\mathcal{S}_i^j}\right)$ has Property $(T)$ is at least 
 \begin{align*}
P_{T}(m,d,\lambda_0,j)(i)+P_{sc}(m,d,\alpha,\beta,j)(i)P_{qi}(m,d,\alpha,\beta,j)(i)-1.
 \end{align*}

There exists $j_0$ so that for all $j\geqslant j_0$ we have that 
  \[d^{2(C+1)}(2m)^{-(1-{\eta(m)})\beta j}<1\text{ and }d^{2(C+1)}(2m)^{-(1-{\eta(m)})2j\alpha\left( \frac{1-{\eta(m)}}{\eta(m)} - \beta\right)}<1.\]

 Then,  $P_{sc}P_{qi}$ converges to $1$  exponentially as $i\to \infty$.  
 Simultaneously, the probability $P_T$ converges to $1$  exponentially as  $i\to \infty$.
 
 A graph satisfying the $Gr'(\alpha)$--small cancellation condition clearly satisfies the $Gr'(\alpha')$-- condition for all $\alpha'\geqslant \alpha$. Theorem \ref{T: ap2} follows.

 \subsection{Proof of Theorem \ref{T: ap1} and  Theorem \ref{T: ap3}}\label{S: proof T}

We extend the proof from \cite{ollivier_kazhdan_2007}, in particular, Lemmas \ref{lqi} and \ref{scwl} to the free product setting. We view $\widetilde{\mathcal{G}^j}$ with the edge length and $\fold(\widetilde{\mathcal{G}^j})$ with the free product length over $\langle a_1\rangle*\ldots*\langle a_m \rangle$.  The probability that the map \emph{folding} $\colon \widetilde{\mathcal{G}^j}\to \fold\left(\widetilde{\mathcal{G}^j}\right)$ is a local quasi-isometric embedding is denoted by $P_{qi}^*$, and  the \emph{conditional} probability that the labeling of $\fold\left(\widetilde{\mathcal{G}^j}\right)$ satisfies the $Gr_*'(\alpha)$--small cancellation condition over  $\langle a_1\rangle*\ldots*\langle a_m \rangle$, under the condition that the folding is a local quasi-isometric embedding, is  denoted by $P_{sc}^*$. That is, the probability that the labeling of $\fold\left(\widetilde{\mathcal{G}^j}\right)$ satisfies the $Gr_*'(\alpha)$--small cancellation condition is at least $P_{qi}^*P_{sc}^*$.

We derive  lower bounds for this probabilities. Our results then require a careful analysis of the obtained estimates.

  \begin{lemma} \label{L: uniformly chosen random word} Let $W_l$ be a word of length $l$ in $2m$ letters chosen uniformly at random. Then 
  \[P(|W_l|_*\leqslant L)\leqslant (2m-1)^{\frac L2}\left( \frac{lm}{2m-1}\right)^{\frac12}\left(\frac{\sqrt2}{(2m)^{1-{\eta(m)}}}\right)^l.\]
  \end{lemma}
 
 \begin{proof}
 Let $B_l$ be the ball of radius $l$ with respect to the word length metric in the free group on $m$ generators. Let $p^l_x$ denote the probability that $W_l=x$ where the equality is in the free group.
 
 The number of elements $x$ in $B_l$ such that $|x|_*=k$ is at most
 
  \[\sum_{k\leqslant l' \leqslant l}\binom{l'-1}{k-1}2m(2m-2)^{k-1}\leqslant l\binom{l-1}{k-1} 2m(2m-1)^{k-1}.\]
 
   Hence 
   \begin{align}
   \sum_{x\in B_l}(2m-1)^{-|x|_*}&\leqslant \sum_{1\leqslant k \leqslant l}l\binom{l-1}{k-1}2m(2m-2)^{k-1}(2m-1)^{-k} \label{a1}\\
   &\leqslant l \frac{2m}{2m-1} 2^{l-1}. \nonumber
   \end{align}
   We compute the expected value,
   \begin{align*}
   \mathbb{E}\left((2m-1)^{-\frac12|W_l|_*}\right)&=\sum_{x\in B_l}(2m-1)^{-\frac12 |x|_*} p^l_x.
   \end{align*}
   By the Cauchy-Schwartz inequality, this is bounded by
   \begin{align*}
   & \leqslant \left(\sum_{x\in B_l}(2m-1)^{|x|_*}\right)^{\frac12}\left(\sum_{x\in B_l}(p^l_x)^2\right)^{\frac12}.
   \end{align*} 
   The right term $\sum_{x\in B_l}(p^l_x)^2$ is the return probability of the simple random walk on the free group of rank $m$ at time $2l$. This probability is at most $(2m)^{-(1-{\eta(m)})2l}$. 
  Applying inequality \eqref{a1},  we have that 
  \begin{align*}
   \mathbb{E}\left((2m-1)^{-\frac12|W_l|_*}\right)&\leqslant  \left(l \frac{2m}{2m-1} 2^{l-1}\right)^{\frac12}(2m)^{-(1-{\eta(m)})l}\\
   &= \left(\frac{lm}{2m-1} \right)^{\frac12}\left(\frac{\sqrt2}{(2m)^{1-{\eta(m)}}}\right)^l.
  \end{align*} 
The result now follows using Markov's inequality,
\begin{align*}
 P(|W_l|_*\leqslant L)&=P\left((2m-1)^{-L/2} \leqslant(2m-1)^{-\frac12|W_l|_*}\right)\\
 &\leqslant (2m-1)^{-L/2} \mathbb{E}\left((2m-1)^{-\frac12|W_l|_*}\right).
\end{align*}

 \end{proof}
 
 \begin{lemma}\label{L: fpqi} For all  $m\geqslant 2$, $d\in \mathbb{N}$,  $\beta>0$,  $j\geqslant 1$, if $\deg(v) \leqslant d$ for all $v\in V(\mathcal{G})$, then the folding map from   $\widetilde{\mathcal{G}^j}$, equipped with the edge length, to  $\fold\left(\widetilde{\mathcal{G}^j}\right)$, equipped with the free product metric in $\langle a_1\rangle* \ldots *\langle a_m\rangle$,  is a $\left(\frac1{2(1-{\eta(m)})},\beta \g(\mathcal{G})j,\g(\mathcal{G})j\right)$--local quasi-isometric embedding with probability  $P_{qi}^*$, which is
 \[\geqslant 1-(jd)^2d^{\diam(\mathcal{G})+\g(\mathcal{G})}\left(\g(\mathcal{G}) j\frac {m}{2m-1}\right)^{\frac12}2^{\frac {\g(\mathcal{G})j}2}\left(\frac{\sqrt2 (2m)^{{\eta(m)}}}{2m}\right)^{\beta \g(\mathcal{G})j}.\]
 \end{lemma}
 \begin{proof} Let $g:=\g(\mathcal{G})$. 
  Choose a path $p$ of edge length $\beta gj+\ell\leqslant gj$ in $\widetilde{\mathcal{G}^j}$. It suffices to show that $|\ell(p)|_*>2(1-{\eta(m)})\ell$, where $\vert \cdot \vert_*$ denotes the free product length on the folded graph $\fold\left(\widetilde{\mathcal{G}^j}\right)$. 
  
  The probability that a random labeling $\omega$ of $p$, i.e. a word in $\beta gj+\ell$ letters, chosen uniformly at random, has the free product length at most $2(1-{\eta(m)})\ell$ has been estimated  in Lemma \ref{L: uniformly chosen random word}. It is at most
  \[
  (2m-1)^{(1-{\eta(m)})\ell}\left( \frac{(\beta gj+\ell)m}{2m-1}\right)^{\frac12}\left(\frac{\sqrt2}{(2m)^{1-{\eta(m)}}}\right)^{\beta gj+\ell}.
  \]
  
There are at most $(jd)^2d^{\diam(\mathcal{G})+\g(\mathcal{G})}$ paths of length $\leqslant gj$ in $\mathcal{G}^j$. Indeed, there are at most $d^{\diam(\mathcal{G})}$ starting vertices for a simple path in $\mathcal{G}$. There are at most $d^{\diam(\mathcal{G})+l}$ possibilities of paths of length $\leqslant l$ in $\mathcal{G}$.  A path in $\mathcal{G}^j$ of edge length $l'$ is traveling along $l'/j$ vertices in $\mathcal{G}$ with at most $jd$ possibilities to choose the starting/terminal vertex. Therefore, there are at most $(jd)^2d^{\diam(\mathcal{G})+l'/j}$ possibilities for paths of length $l'$ in $\mathcal{G}^j$.
  
  We combine both estimates to complete the proof.
 \end{proof}

Compared to the estimate of $P_{qi}$ in Lemma \ref{lqi}, we have a new subexponential term and a new exponential term $2^{{\g(\mathcal{G})j}/2}2^{\beta\g(\mathcal{G})j/2}$ in our estimate of $P^*_{qi}$. To obtain the required results we therefore need a more careful analysis than above.

  \begin{pr}\label{P: fix beta} For all $m> 35$, $d\in \mathbb{N}$, there is $j_0>0$ such that for all  $j> j_0$ the folding from  $\widetilde{\mathcal{S}_i^j}$, equipped with the word length metric, to  $\fold\left(\widetilde{\mathcal{S}_i^j}\right)$, equipped with the free product metric in $\langle a_1\rangle* \ldots *\langle a_m\rangle$,  is a $\left(\frac1{2(1-{\eta(m)})},1/3  \g(\mathcal{G})j,\g(\mathcal{G})j\right)$--local quasi-isometric embedding with probability tending to $1$ exponentially as $i\to \infty$.
 \end{pr}
 Compared to Lemma \ref{L: fpqi}, we have specified $\beta=1/3$.
 \begin{proof}
 The claim follows when $$2^{\frac12}\left(\frac{\sqrt2 (2m)^{{\eta(m)}}}{2m}\right)^{1/3}<1.$$ Then there is $j_0$ such that for all $j>j_0$ we have that  \[d^{(C+1)}\left(2^{\frac12}\left(\frac{\sqrt2 (2m)^{{\eta(m)}}}{2m}\right)^{1/3}\right)^j<1. \]
By Lemma \ref{L: fpqi}, $P^*_{qi}\to 1$ exponentially as $i\to \infty$.
 
 By a result of Kesten \cite{kesten_symmetric_1959}*{Th. 3}, $(2m)^{{\eta(m)}}=2\sqrt{2m-1}$, so we need that 
 $$\left(\frac{\sqrt2 \sqrt{2m-1}}{m}\right)^{1/3}<\frac1{\sqrt{2}}.$$
 That is,
 $0<m^2-32 m +16.$ 
This holds for all $m> 35$. Note that 
${\eta(m)} <\frac23$ if $m>35$. We therefore have
\[
1-{\eta(m)}>\frac13.
\]
We conclude as before.
   \end{proof}

 \begin{lemma}\label{L: scfp}For all $m\geqslant 2$, $d\in \mathbb{N}$,  $\beta$ such that $1-{\eta(m)}-\beta>0$, $\alpha>0$ such that $\alpha<(1-{\eta(m)})-\beta$, $j\geqslant1$, if
\begin{enumerate}
 \item  the vertices of the graph $\mathcal{G}$ have degree at most $d$,
  \item folding $\colon \widetilde{\mathcal{G}^j}\to \fold\left(\widetilde{\mathcal{G}^j}\right)$ is a $\left(\frac1{2(1-{\eta(m)})}, \beta \g(\mathcal{G})j,\g(\mathcal{G})j\right)$--local quasi-isometric embedding, where $\fold\left(\widetilde{\mathcal{G}^j}\right)$ is with the free product metric in $\langle a_1\rangle*\ldots *\langle a_m\rangle$,
  \end{enumerate}   
then the labeling of $\fold\left(\widetilde{\mathcal{G}^j}\right)$ satisfies the $Gr_*'(\alpha)$--small cancellation condition over $\langle a_1\rangle*\ldots *\langle a_m\rangle$ with probability 
\[
P^*_{sc}\geqslant 1-(jd)^4d^{2\diam(\mathcal{G})+2\g(\mathcal{G})}(2m)^{-(1-{\eta(m)})2\g(\mathcal{G})j\alpha\left( 2(1-{\eta(m)}) - \beta\right) }.
\]
   \end{lemma}
\begin{proof} Let $g:=\g(\mathcal{G})$. 
First observe that  by the quasi-isometry assumption 
    \[
   \bar{g}:=\min|\text{labels of non-trivial cycles in } \fold(\widetilde{\mathcal{G}^j})|_*\geqslant (2(1-{\eta(m)})-\beta)gj.                 
           \]
 
Let $\alpha<(1-{\eta(m)})-\beta$. It suffices to estimate the probability that there are no $\alpha$-pieces, that is, pieces $p$  in $\fold(\widetilde{\mathcal{G}^j})$ such that $|\ell(p)|_{*}=\alpha\bar{g}$.
 Let $q,q'$ be immersed paths in $\widetilde{\mathcal{G}^j}$ whose folding equals $p$. For the word length $|\ell(q)|,|\ell(q')|\leqslant \frac{gj}2$. Indeed, otherwise, $|\ell(p)|_*>((1-{\eta(m)})-\beta)gj$ by the quasi-isometry assumption, a contradiction. On the other hand, $|\ell(q)|,|\ell(q')|\geqslant |\ell(p)|\geqslant |\ell(p)|_*=\alpha \bar{g}$.

 Therefore, suppose that $q$ and $q'$ are $\alpha \bar{g}$ pieces in $\widetilde{\mathcal{G}^j}$. We now apply the following.
 \begin{lemma}[\cite{ollivier_kazhdan_2007}*{Prop. 7.11}]
  Let $q,q'$ be two immersed paths in a graph $\mathcal{G}$ of girth $g$. Suppose that $q$ and $q'$ have length $l$ and $l'$ respectively, with $l$ and $l'$ at most $g/2$. Endow $\mathcal{G}$ with a random labeling. Suppose that after folding the graph, the paths $q$ and $q'$ are mapped to distinct paths. Then the probability that $q$ and $q'$ are labeled by two freely equal words is at most \[C_{l,l'}(2m)^{-(1-{\eta(m)})(l+l')},\] where $C_{l,l'}$ is a term growing subexponentially in $l+l'$.
 \end{lemma}

Hence, the probability that two paths $q$, $q'$ in $\widetilde{\mathcal{G}^j}$ are $\alpha \bar{g}$ pieces in $\widetilde{\mathcal{G}^j}$ is at most  
 \[C_{gj}(2m^{-(1-{\eta(m)})2gj\alpha(2(1-{\eta(m)})-\beta)}),\] where $C_{gj}$ is a sub-exponential term in $g$. 
 The probability that there are two such paths $q$, $q'$ in $\widetilde{\mathcal{G}^j}$ is at most
 $$j^4 v^{2\diam(\mathcal{G})+2g}C_{gj}(2m^{-(1-{\eta(m)})2gj\alpha(2(1-{\eta(m)})-\beta)}).$$
\end{proof}

\begin{them}\label{small cancellation conditions are generically equivalent}
For all $m>64$, $d\in \mathbb{N}$, there is $j_0>0$ such that for all positive numbers $j>j_0$, the labeling of $\fold{\left(\widetilde{\mathcal{S}_i^j}\right)}$ satisfies both
 \begin{itemize}
 \item the $Gr_*'(\alpha)$--small cancellation condition  over $\langle a_1\rangle * \ldots * \langle a_m\rangle$, and
 \item the $Gr'(\alpha)$--small cancellation condition with respect to the word length metric in the free group on $a_1,\ldots,a_m$,
 \end{itemize}
with probability tending to $1$ exponentially as $i\to \infty$.
\end{them}

Observe that $\mathcal{S}$ does not need to satisfy condition \eqref{as: spectral} above.

 \begin{proof} Note that $1-\eta(m)>1/3$,  choose $\beta=1/4$ so that $1-\eta(m)>  \beta$, and  $\frac{1-\eta(m)}{2\eta(m)}-\beta >0$.  If $m\geqslant 64$, then, by an estimate as in the proof of Proposition~\ref{P: fix beta}, $$2^{\frac12}\left(\frac{\sqrt2 (2m)^{{\eta(m)}}}{2m}\right)^{1/4}<1.$$
  Choose $\alpha>0$ such that 
    \[
    \alpha<  \frac{1-{\eta}}{2\eta} - 1/4 \text{ and } \alpha<(1-{\eta})-1/4.
    \]
    
  The probability that $\fold\left({\widetilde{\mathcal{S}_i^j}}\right)$ does not satisfy the required small cancellation conditions is at most  
  \begin{align*}
 1-P_{sc}(m,d,\alpha,1/4,j)(i)P_{qi}&(m,d,\alpha,1/4,j)(i)+1\\&-P_{sc}^*(m,d,\alpha,1/4,j)(i)P^*_{qi}(m,d,\alpha,1/4,j)(i).  
  \end{align*}

   There exists $j_0$ such that for all $j>j_0$ we have that 
  \begin{itemize}
   \item $d^{2(C+1)}(2m)^{-(1-{\eta(m)})1/4 j}<1,$ 
 \item   $d^{2(C+1)}(2m)^{-(1-{\eta(m)})2j\alpha\left( \frac{1-{\eta(m)}}{\eta(m)} - 1/4\right)}<1,$ 
   \item $d^{(C+1)}\left(2^{\frac12}\left(\frac{\sqrt2 (2m)^{{\eta(m)}}}{2m}\right)^{1/4}\right)^j<1,$ and 
   \item $(jd)^4d^{2(C+1)}(2m)^{-(1-{\eta(m)})2j\alpha\left( 2(1-{\eta(m)}) - 1/4 \right)}<1.$
   \end{itemize}
   Then
   \[
         P_{sc}(m,d,\alpha,1/4,j)(i)P_{qi}(m,d,\alpha,1/4,j)(i)
        \]
 and simultaneously 
 \[
  P_{sc}^*(m,d,\alpha,1/4,j)(i)P^*_{qi}(m,d,\alpha,1/4,j)(i)
 \]
 tend to 1 exponentially as $i\to \infty$.

 \end{proof}

  Theorems \ref{T: ap1} and \ref{T: ap3} now follow as Theorem \ref{T: ap2}.

\end{document}